\documentclass[11pt]{amsart}
\usepackage{amssymb, latexsym, amsmath, amsfonts, hyperref, enumitem, fancyhdr}
\usepackage[mathscr]{eucal}

\newtheorem{thm}{Theorem}[section]

\newtheorem{lem}[thm]{Lemma}

\theoremstyle{definition}
\newtheorem{defn}[thm]{Definition}

\theoremstyle{remark}
\newtheorem{rem}[thm]{Remark}
\numberwithin{equation}{section}

\hypersetup{
  colorlinks,
  citecolor=black,
  linkcolor=black,
  urlcolor=blue}
\theoremstyle{plain}

  {\end{enumerate}}

\newcommand{\D}{\mathbb{D}}
\newcommand{\C}{\mathbb{C}}

\newcommand{\R}{\mathbb{R}}
\newcommand{\h}{\mathbb{H}}

\newcommand{\N}{\mathbb{N}}
\newcommand{\T}{\mathbb{T}}

\hypersetup{
  colorlinks,
  citecolor=blue,
  linkcolor=blue,
  urlcolor=blue}

\begin{document}
\title{Operator Space Structures on $\ell^1(n)$}
\thanks{The first named author was supported by the NBHM, Government of India.}
\thanks{The second named author was supported by CSIR, Government of India.}

\keywords{Operator Spaces, Finite dimensional embedding, MIN structure}

{\let\thefootnote\relax\footnote{The results in this paper are from the first author's thesis ``The Carath\'{e}odory-Fej\'{e}r Interpolation Problems and the von-Neumann Inequality" submitted to the Indian Institute of Science, Bangalore-560012.}}

\author{Rajeev Gupta}
\author{Md. Ramiz Reza}
\address{Rajeev Gupta: Department of Mathematics, Indian Institute of Science, Bangalore}
\email{rajeev10@math.iisc.ernet.in}

\address{Md. Ramiz Reza: Department of Mathematics, Indian Institute of Science, Bangalore}
\email{ramiz@math.iisc.ernet.in}

\pagestyle{headings}

\begin{abstract}
We show that the complex normed linear space 
$\ell^1(n),$  $n>1,$ 
has no isometric embedding into 
$k\times k$ complex matrices for any $k\in \N$
and discuss a class of 
infinite dimensional operator space structures on it.
\end{abstract}

\maketitle

\section{Introduction}
In this paper, all the normed linear spaces considered are over the field of complex numbers unless specified. It is well known that there are isometric embeddings of real $\ell^1(n)$ into real $\ell^\infty(k)$ for some $k$ and hence into the space of $k\times k$ real matrices $M_k(\R).$ 
%(\textcolor{red}{give reference})
However, we prove that $\ell^1(n),\,n>1,$ 
has no isometric embedding into $M_k$ 
for any $k\in \N.$ 
This shows that there is no operator space structure 
on $\ell^1(n),\,n>1,$ which can be induced by 
any $k\times k$ matrices $A_1,\ldots,A_n.$ 
Furthermore, we study the operators space structures on $\ell^1(n)$.  
We recall some definitions first. 
 \begin{defn}
	An abstract operator space 
	is a normed linear space $V$ 
	together with a sequence of  
	norms $\|\cdot\|_k$ defined on the linear space 
	$$M_k(V):=\big \{\big (\!\! \big (v_{ij} \big )\!\! \big )\vert v_{ij} \in V, 1\leq i,j \leq k \big \},\,\, \forall k\in \mathbb N,$$	
	with the understanding that 
	$\|\cdot\|_1$ is the norm of $V$ and 
	the family of norms $\|\cdot\|_k$ 
	satisfies the compatibility conditions:
	\begin{itemize}
		\item[1.] $\|T\oplus S\|_{p+q} = \max\big\{\|T\|_p,\|S\|_q\big\}$
		and
		\item[2.] $\|ASB\|_q \leq \|A\|_{op} \|S\|_p \|B\|_{op}$
	\end{itemize} 
	for all $S\in M_q(V),T\in M_p(V),$ 	
	$A\in M_{q\times p}(\C)$ and $B\in M_{p\times q}(\C).$
\end{defn}

Let $(V,\|\cdot\|_k )$ and 
$(W,\|\cdot\|_k)$ be two operator spaces. 
A linear bijection 
$T:V\to W$ is said to be a complete isometry 
if $T\otimes I_k:(M_k(V),\|\cdot\|_k)\to (M_k(W),\|\cdot\|_k)$
is an isometry for every $k\in \N$. 
Operator spaces $(V,\|\cdot\|_k )$ and 
$(W,\|\cdot\|_k)$ are said to be completely isometric
if there is a linear complete isometry $T:V\to W$. 
A  well known theorem of Ruan 
says that any operator space $(V,\|\cdot\|_k)$ 
can be embedded, completely isometrically, 
into $C^*$-algebra $\mathcal{B}(\h)$ 
for some Hilbert space $\h$.
There are two natural operator space structures
on any normed linear space $V,$ which may coincide.  
These are the MIN and the MAX 
operator space structures defined below.
  
\begin{defn}[MIN]
	The MIN operator space structure denoted by MIN($V$) 
	on a normed linear space $V$ 
	is obtained by the isometric embedding of $V$ 
	into the $C^*$-algebra $C((V^*)_1)$, 
	the space of continuous functions 
	on the unit ball $(V^*)_1$ 
	of the dual space $V^*$. 
	Thus for $\big (\!\! \big (v_{ij} \big )\!\! \big )$ in $M_k(V)$, 
	we set
	\[\big\|\big (\!\! \big (v_{ij} \big )\!\! \big )\big\|_{MIN} = \sup\left\{\big\|\big (\!\! \big (f(v_{ij}) \big )\!\! \big )\big\| : f\in (V^*)_1\right\},\]
	where the norm of a scalar matrix 
	$\big (\!\! \big (f(v_{ij}) \big )\!\! \big )$ 
	is the operator norm in $M_k$. 
\end{defn}

\begin{defn}[Max]
	Let $V$ be a normed linear space and  
	$\big (\!\! \big (v_{ij} \big )\!\! \big )\in M_k(V)$. 
	Define
	\[\big\|\big (\!\! \big (v_{ij} \big )\!\! \big )\big\|_{MAX} = \sup\left\{\big\|\big (\!\! \big (Tv_{ij} \big )\!\! \big )\big\|: T:V\to \mathcal{B}(\h)\right\},\]
	where the supremum is taken over 
	all isometry maps $T$ and all Hilbert spaces $\h$. 
	This operator space structure is denoted by MAX($V$).
\end{defn}

These two operator space structures 
are extremal in the sense 
that for any normed linear space $V$, 
MIN($V$) and MAX($V$) are 
the smallest and the largest 
operator space structures on $V$  respectively. 
For any normed linear space $V,$ 
Paulsen \cite{VP} associates a constant, 
namely, $\alpha(V),$ which is defined as following. 
\[\alpha(V):=\sup \left\{\|I_V\otimes I_k\|_{(M_k(V), \|\cdot \|_{\rm MIN}) \to (M_k(V), \|\cdot \|_{\rm MAX})}: k\in \mathbb N \right \}. \]
The constant $\alpha(V)$ is equal to $1$ 
if and only if 
$V$ has only one operator space structure on it. 
There are only a few examples of  normed linear spaces 
for which $\alpha(V)$ is known to be $1.$  These include   
$\alpha(\ell^{\infty}(2))=\alpha(\ell^1(2))=1.$ 
In fact, it is  known (cf. \cite[Page 77]{GP}) 
that $\alpha(V)>1$ if $\dim(V)\geq 3$.

The map $\phi:\ell^\infty(n)\to \mathcal{B}(\C^n)$ 
defined by $\phi(z_1,\ldots,z_n)={\rm diag}(z_1,\ldots,z_n),$  is an isometric embedding of the normed linear space $\ell^{\infty}(n)$ 
into the finite dimensional $C^*-$algebra $\mathcal{B}(\C^n).$ Clearly, this is the MIN structure of the normed linear space $\ell^\infty(n).$ 
We shall, however prove that 
there is no such finite dimensional isometric embedding for 
the dual space $\ell^1(n).$ 
Nevertheless, we shall construct, explicitly, a class of isometric infinite dimensional embeddings of $\ell^1(n)$. 
Unfortunately, all of these embeddings are completely isometric to the MIN structure.  In the end of this paper, using these embeddings and Parrott's example in \cite{GMParrott}, we construct an operator space structure on $\ell^1(3),$ which is distinct from the MIN structure.

\section{\texorpdfstring{$\ell^{1}(n)$}{TEXT} has no isometric embedding into any \texorpdfstring{$M_k$}{TEXT}}
In this section, we will show that 
there does not exist 
an isometric embedding of $\ell^{1}(n),$ $n>1,$ into any finite dimensional matrix algebra $M_k,$ $k\in\mathbb N.$ 
Without loss of generality, 
we prove this for the case of  $n=2.$ For the proof of the main theorem of this section, we shall need the following lemma. 

\begin{lem}\label{Reason for non finite dimensinality}
	For $k\in\N$ and 
	$\theta_1,\ldots,\theta_k\in [0,2\pi),$ 
	there exists $a_1,a_2\in\C$ 
	such that 
	\[\max\limits_{j=1,\ldots, k}\big|a_1 + e^{i\theta_j}a_2\big|< \big|a_1\big| + \big|a_2\big|.\]
\end{lem}
\begin{proof} For any two non-zero complex numbers $a_1,a_2,$ we have  
\[\max\limits_{j=1,\ldots, k}\big|a_1 + e^{i\theta_j}a_2\big|
=\max\limits_{j=1,\ldots, k}\big||a_1| + e^{i(\theta_j+\phi_2-\phi_1)}|a_2|\big|,\] 
where $\phi_1$ and $\phi_2$ are the arguments of 
$a_1$ and $a_2$ respectively. 
Setting $\alpha_j=\theta_j+\phi_2-\phi_1,$ we have 
\begin{align*}
\max\limits_{j=1,\ldots, k}\big|a_1 + e^{i\theta_j}a_2\big|^2
&=\max\limits_{j=1,\ldots, k}\big||a_1| + e^{i\alpha_j}|a_2|\big|^2\\
&=\max\limits_{j=1,\ldots, k}\big||a_1|^2 + |a_2|^2+ 2|a_1a_2|cos\alpha_j\big|.
\end{align*}
Therefore
\[\max\limits_{j=1,\ldots, k}\big|a_1 + e^{i\theta_j}a_2\big|
=\big|a_1\big|+\big|a_2\big|
\] 
if and only if 
$\cos \alpha_j=1$ for some $j,$ 
that is, if and only if 
$\alpha_j=0$ for some $j$. 
Choose $a_1$ and $a_2$ such that 
$\phi_1-\phi_2 \neq \theta_j$ 
for all $j=1,\ldots, k$. 
The existence of such a pair $a_1$ and $a_2$ 
proves  the lemma. 
\end{proof}
The referee points out that the lemma is equivalent to the statement  
``There is no isometric embedding of $\ell^1(n),\,n>1,$ 
into $\ell^\infty(k)$ for any $k\in\N$." %\textcolor{red}{Reference}, 
The argument below validates this equivalence. 
  
Suppose $S:\ell^1(2)\to \ell^\infty(k)$ defined by 
$S(z_1,z_2):=(a_1z_1+b_1z_2,\ldots,a_kz_1+b_kz_2)$ 
is an isometry with smallest possible $k\in\N.$ 
Then, due to the minimality of $k,$ 
it follows that $|a_j|=|b_j|=1$ for $j=1,\ldots,k.$ 
Without loss of generality, we can assume that $a_j=1$ for $j=1,\ldots,n.$  
Using Lemma \ref{No finite dimensional embedding},  
we conclude that $S$ can not be an isometry. 
For the converse part, we note that Lemma \ref{No finite dimensional embedding} 
is equivalent to the statement that the linear map  
$S:\ell^1(2)\to \ell^\infty(k)$ defined by 
$S(z_1,z_2):=(z_1+e^{i\theta_1}z_2,\ldots,z_1+e^{i\theta_k}z_2)$ 
can not be an isometry.

Now, we prove the main theorem of this section.
\begin{thm}\label{No finite dimensional embedding}
	There is no isometric embedding of $\ell^{1}(2)$ into $M_k$ for any $k\in \N.$.
\end{thm}
\begin{proof}
Suppose there is a $k-dimensional$ isometric embedding $\phi$ of $\ell^{1}(2)$. 
Then $\phi$ is induced by a pair of operators $T_1, T_2\in M_k$ of norm $1,$ defined by the rule,    
$\phi(a_1,a_2)=a_1T_1 + a_2T_2.$  
Let $U_1$ and $U_2$ in $M_{2k}$ be the pair of unitary maps: 
$$
U_i:= \begin{pmatrix}T_i & D_{T_i^*} \\ D_{T_i} & -T_i^*, \end{pmatrix} i=1,2,
$$
where $D_{T_i}$ is the positive square root of the (positive) operator $I-T_i^* T_i.$   
Now, we have 
\[P_{\C^k}(a_1U_1 + a_2U_2)_{|\C^k}=a_1T_1 + a_2T_2.\]
(This dilating pair of unitary maps is not necessarily commuting nor is 
 it a power dilation!)  
Thus 
$\psi:\ell^{1}(2)\to M_{2k}(\C)$ 
defined by 
$\psi(a_1,a_2)=a_1U_1 + a_2U_2$
is also an isometry. 
Since norms are preserved under unitary operations, without loss of generality,    
we assume $U_1=I$ and $U_2$ 
to be a diagonal unitary, say, $D.$ 
Let 
$D={\rm diag}\big(e^{i\theta_1},\ldots,e^{i\theta_{2k}}\big)$. 
Applying Lemma \ref{Reason for non finite dimensinality}, 
we obtain complex numbers $a_1$ and $a_2$ 
such that 
\[\max\limits_{j=1,\ldots, 2k}\big|a_1 + e^{i\theta_j}a_2\big|
< \big|a_1\big| + \big|a_2\big|.\]
Hence $\psi$ cannot be an isometry, which contradicts the 
hypothesis that 
$\phi$ is an isometry.    
\end{proof}

\begin{rem}
Let $X$ be a finite dimensional normed linear space. Suppose $X$ is embedded isometrically in $M_k$ for some $k\in \N$, then the standard dual operator space structure on $X^*$ need not admit an embedding in $M_n$ for any $n\in \N$.
\end{rem}
\begin{rem}
	An amusing corollary to this theorem is that the two spaces $\ell^{\infty}(n)$ and $\ell^{1}(n)$ 
	cannot be isometrically isomorphic for $n>1.$
\end{rem}

\begin{rem}
	Prof. G. Pisier points out that Theorem \ref{No finite dimensional embedding} may be true if one replaces $M_k$ by $\mathcal{K}(\h)$, the set of all compact operators on an infinite dimensional separable Hilbert space $\h$. 
\end{rem}

\section{Infinite dimensional embeddings of \texorpdfstring{$\ell^{1}(n)$}{TEXT}}
In this section we construct operator space structure on $\ell^1(n)$ $(n\geq 3)$ which is not completely isometric to MIN structure of $\ell^1(n).$
%First we find a class of embeddings of $\ell^1(n)$ into $\mathcal{B}(\h).$ All of them induce operator space structures on $\ell^1(n)$ which turns out to be completely isometric to MIN structure on $\ell^1(n).$ Then we combine this with a result of Parrot to construct the operator space structure on $\ell^1(n)$ which is different from MIN structure on $\ell^1(n).$   

Let $\h_i$ be a Hilbert space and $T_i$ 
be a contraction on 
$\h_i$ for $i=1,\ldots,n.$ 
Assume that the unit circle $\T$ is contained in $\sigma(T_i),$ the spectrum of $T_i,$ for $i=1,\ldots,n$. 
Denote 
$$\tilde{T_1}=T_1\otimes I^{\otimes (n-1)},
\tilde{T_2}=I\otimes T_2\otimes I^{\otimes (n-2)},\ldots ,
\tilde{T_n}=I^{\otimes (n-1)}\otimes T_n$$
and $\boldsymbol{T}=(\tilde{T_1},\ldots,\tilde{T_n}).$

\begin{thm}\label{Infiniteoperator}
Suppose the operators $\tilde{T}_1,\ldots,\tilde{T}_n$ are defined as above. Then, the function 
$$f_{\boldsymbol{T}}:\ell ^1(n)\to \mathcal B(\h_1 \otimes\cdots \otimes \h_n)$$
defined by $f_{\boldsymbol{T}}(a_1,\ldots,a_n):=a_1\tilde{T_1} +\cdots + a_n \tilde{T_n}$
is an isometry.
\end{thm}
\begin{proof}
Since 
$\mathbb{T}\subset \sigma (T_i)$ and 
$T_i$ is a contraction for $i=1,\ldots,n$, 
it follows that 
$\mathbb{T}\subset \partial\sigma (T_i)$ 
for $i=1,\ldots,n.$ 
From (cf. \cite[Proposition 6.7, Page 210]{Conway}), we have 
$\mathbb{T}\subset \sigma _a(T_i)$ for $i=1,\ldots,n,$ 
where $\sigma _a(T_i)$ is 
the approximate point spectrum of $T_i.$
Thus for any $i\in\{1,\ldots,n\}$ and $\lambda \in \mathbb{T},$ 
there exists a sequence of unit vectors 
$(x_{m}^{i})_{m\in\mathbb{N}}$ 
in $\h_i$ 
such that 
$$\| (T_i-\lambda)(x_{m}^{i})\| \longrightarrow 0\mbox{ as }  m\longrightarrow \infty.$$ 
Now, applying the Cauchy-Schwarz's inequality, 
we have 
\begin{align*}
	|\langle (T_i-\lambda)(x_{m}^{i}),(x_{m}^{i})\rangle |&
	\leq  \| (T_i-\lambda)(x_{m}^{i})\| \| (x_{m}^{i}) \| \\
	&= \| (T_i-\lambda)(x_{m}^{i})\| \longrightarrow 0   
\end{align*}
as $m\longrightarrow \infty.$ Hence  
$\langle T_i(x_{m}^{i}),(x_{m}^{i})\rangle \longrightarrow \lambda$ as $m\longrightarrow \infty.$
Let $(a_1,\ldots ,a_n)$ 
be any vector in $\ell ^1(n)$ such that none of its co-ordinates is zero.
Let $\lambda _1=e^{-i\arg(a_1)},\lambda _2=e^{-i\arg(a_2)},\ldots ,\lambda _n=e^{-i\arg(a_n)}.$
% be corresponding complex numbers in $\mathbb{T}$. 
 Now for each $i\in \{1,\ldots ,n\}$, 
 we have $(x_{m}^{i})_{m\in\mathbb{N}}$, 
 a sequence of unit vectors from  $\h_i$, 
 such that 
 $$\langle T_i(x_{m}^{i}),(x_{m}^{i})\rangle \longrightarrow \lambda _i\mbox{ as } m\longrightarrow \infty .$$
As $m$ goes to $\infty$, we have  
%\small{
\begin{align*}
	&\big|\langle (a_1T_1\otimes I^{\otimes (n-1)}+\cdots 
	+ a_nI^{\otimes (n-1)}\otimes T_n)(x_{m}^{1}\otimes\cdots\otimes x_{m}^{n}),
	(x_{m}^{1}\otimes \cdots\otimes x_{m}^{n}) \rangle \big|\\
	 &= \big|a_1\langle T_1 (x_{m}^{1}),(x_{m}^{1}))\rangle + 
	 \cdots + a_n\langle T_n (x_{m}^{n}),(x_{m}^{n}))\rangle \big|
	\longrightarrow  \big|a_1 \lambda _1 + \cdots + a_n \lambda _n\big|\\ 
	& =  |a_1| +  \cdots + |a_n| =  \| (a_1,\ldots ,a_n)\| _1.
\end{align*} 
%}
Hence 
$\| a_1\tilde{T_1} +\cdots + a_n \tilde{T_n}\| \geq \| (a_1,\ldots ,a_n)\| _1.$
Also 
	$$\|a_1\tilde{T_1}+ \cdots + a_n \tilde{T_n}\|
	\leq  |a_1|\| T_1 \| + \cdots + |a_n|\| T_n \|.$$
Hence 
$ \|a_1\tilde{T_1}+ \cdots + a_n \tilde{T_n} \| = \| (a_1,\ldots ,a_n)\| _1,$
proving that $f_{\boldsymbol{T}}$ is an isometry.

If some of the co-ordinates in the vector $(a_1, \ldots , a_n)$ are zero, the same argument, as above, remains valid after dropping those co-ordinates. 
\end{proof}

An adaptation of the technique involved in the proof of Theorem \ref{Infiniteoperator},
also proves the following theorem. 
\begin{thm}\label{LessI}
For $i=1,\ldots,n,$ let $T_i$ be a contraction on a Hilbert space $\h_i$ and 
$\T\subseteq \sigma(T_i).$  Denote $\tilde{T}_i=T_1\otimes \cdots\otimes T_i\otimes I_{\h_{i+1}\otimes\cdots \otimes \h_n}$ and $\boldsymbol{T}=(\tilde{T_1},\ldots,\tilde{T_n}).$
Then, the function 
$$g_{\boldsymbol{T}}:\ell ^1(n)\to \mathcal B(\h_1 \otimes\cdots \otimes \h_n)$$
defined by $g_{\boldsymbol{T}}(a_1,\ldots , a_n):=a_1\tilde{T_1} +\cdots + a_n \tilde{T_n}$
is an isometry.
\end{thm}

\begin{rem}
We show that all the operator spaces induced by the isometries defined in Theorem \ref{Infiniteoperator} are completely isometric to the MIN structure. 
%This includes Remark \ref{MINon2} and Remark \ref{MINon3}. 

Suppose $T_1,\ldots,T_n$ are contractions 
on Hilbert spaces $\h_1,\ldots,\h_n$ respectively 
with the property that 
$\T\subseteq \sigma(T_i)$ for $i=1,\ldots,n$.  
Denote $\tilde{T}_1=T_1\otimes I_{\h_2}\otimes\cdots \otimes I_{\h_n},\ldots, 
\tilde{T}_n=I_{\h_1}\otimes \cdots \otimes I_{\h_{n-1}}\otimes T_n.$ 
Then the map $f_{\boldsymbol{T}}$ defined as in the 
Theorem \ref{Infiniteoperator} is an isometry. 
The dilation theorem due to Sz.-Nagy
(cf. \cite[Theorem 1.1, Page 7]{Paulsen}), 
gives unitary maps  
$U_j:\mathbb{K}_j\to \mathbb{K}_j,$ dilating the contraction 
$T_j,$ for $j=1,\ldots,n$. 
The operator space structure defined by the isometry 
$g:\ell^1(n)\to \mathcal{B}(\mathbb{K}_1\otimes \cdots \otimes\mathbb{K}_n),$
where $g(a_1,\ldots,a_n)=a_1U_1\otimes I_{\mathbb{K}_2\otimes\cdots \otimes\mathbb{K}_n}+ \cdots +a_nI_{\mathbb{K}_1\otimes \cdots\otimes\mathbb{K}_{n-1}} \otimes U_n,$ 
is no lesser than that of $f_{\boldsymbol{T}}$. Since $U_1,\ldots,U_n$ are unitary maps, $C^*-$algebra generated by $U_1\otimes I_{\mathbb{K}_2\otimes\cdots \otimes \mathbb{K}_n},\ldots, I_{\mathbb{K}_1\otimes \cdots\otimes\mathbb{K}_{n-1}}\otimes U_n$ is commutative. From (cf. \cite[Proposition 1.10, Page 24]{GP}), we conclude that $g$ is a complete isometry.

It can similarly be shown that all the operator spaces induced by the isometries defined in Theorem \ref{LessI} are completely isometric to the MIN structure.
\end{rem}
\subsection{Operator space structures on \texorpdfstring{$\ell ^{1}(n)$}{TEXT} different from the MIN structure}
Parrott \cite{Pexample} provides an example of 
a contractive homomorphism on $\mathcal{A}(\D^3)$ 
which is not completely contractive. 
(Here $\mathcal{A}(\D^3)$ is the closure, 
with respect to the supremum norm on $\D^3$, 
of the polynomial in 3 complex variables.) 
Using a triple $(I,U,V)$(defined below) 
of $2\times 2$ unitaries, it was shown in \cite{GMParrott} 
that examples due to Parrott may be easily 
thought of as examples of linear contractive maps 
on $\ell^1(3)$ which are not completely contractive. 
Indeed this realization shows that 
the operator space structure on $\ell^1(3)$ can not be unique. 
In this section, using the example from \cite{GMParrott}, 
we give an explicit operator space structure 
$\|\cdot\|_{\texttt{os}}$ on $\ell^1(3)$, 
which is not completely isometric to the MIN structure

% and we show that 
%$\|(I,U,V)\|_{\rm {MAX}}=\|(I,U,V)\|_{\texttt{os}}=3$ 
%but $\|(I,U,V)\|_{\rm{MIN}}<3.$ 

Consider the following $2\times 2$ unitary operators:
\[I_2=\left(
	\begin{array}{cc}
		1 & 0\\
		0 & 1\\
	\end{array}
\right),\,
U:=\left(
	\begin{array}{cc}
		\frac{1}{2} & \frac{\sqrt{3}}{2}\\
		\frac{\sqrt{3}}{2} & -\frac{1}{2}\\
	\end{array}
\right)\,{ \rm and }\,
V:=\left(
	\begin{array}{cc}
		\frac{1}{2} & -\frac{\sqrt{3}}{2}\\
		\frac{\sqrt{3}}{2} & \frac{1}{2}\\
	\end{array}
  \right).
\]
It is clear that the map $h:\ell^1(3)\to M_2,$ 
defined by $h(z_1,z_2,z_3)=z_1I+z_2U+z_3V,$ is of norm at most $1.$ 
The computations done in \cite{GMParrott} includes the following:
\begin{equation}\label{equal3}
\|I\otimes I+U\otimes U+V\otimes V\|=3 
\end{equation}
and 
\begin{equation}\label{less3}
\sup_{z_1,z_2,z_3\in \D}\|z_1 I+z_2 U+z_3 V\|<3.
\end{equation}
Choose a diagonal operator 
$D$ on $\ell ^2(\mathbb{Z})$ 
such that 
$\| D \|\leq 1$ and $\mathbb{T}\subset \sigma(D)$. 
Define 
$$\tilde{T_1}:=\left[
\begin{array}{ll}
      I & 0\\
     0 & D \\
\end{array} 
\right],\tilde{T_2}:=\left[
\begin{array}{ll}
      U & 0\\
     0 & D \\
\end{array} 
\right],\tilde{T_3}:=\left[
\begin{array}{ll}
      V & 0\\
     0 & D \\
\end{array} 
\right]$$
and 
$\hat{T_{1}}=\tilde{T_1}\otimes I\otimes I,\,\hat{T_{2}}=I\otimes\tilde{T_2}\otimes I,\,\hat{T_{3}}=I\otimes I\otimes \tilde{T_n}.$
Let $S_1:=\hat{T_1}\oplus I,\,S_2:=\hat{T_2}\oplus U,\,S_3:=\hat{T_3}\oplus V$
be operators on a Hilbert space $\mathbb{K}$. 
Define 
$S:\ell ^1(3)\longrightarrow B(\mathbb{K})$
by 
$S(e_1)=S_1,\,S(e_2)=S_2,\,S(e_3)=S_3$
and extend it linearly.

From Theorem \ref{Infiniteoperator}, 
we know that the function 
$(z_1,z_2,z_3)\mapsto z_1\hat{T}_1+z_2\hat{T}_2+z_3\hat{T}_3$ is an isometry  
and since $h$ is of norm at most $1,$ 
it follows that the map $(z_1,z_2,z_3)\mapsto z_1S_1+z_2S_2+z_3S_3$ is also an isometry.   
Consequently, there is an operator space structure \texttt{os} on $\ell ^1(3)$ 
for which $S$ is a complete isometry.  Also from \eqref{equal3}, we have
\begin{align*}
	\left \|S_1\otimes I+S_2\otimes U+S_3\otimes V\right\|
	& \geq \left \|I\otimes I+U\otimes U+V\otimes V\right \| =3.
\end{align*}
Thus $\|(I,U,V)\|_{\texttt{os}}=3,$ as norm of $(I,U,V)$ is at most $3$ 
under any operator space structure on $\ell^1(3).$
On the other hand, from \eqref{less3}, we have 
\[\|(I,U,V)\|_{\rm {MIN}}=\sup_{z_1,z_2,z_3\in \D}\|z_1 I+z_2 U+z_3 V\|<3.\]
It follows from \cite[Theorem 14.1]{Paulsen} that if there is a map $\phi:(\ell^1(3),\rm {MIN})\to (\ell^1(3),\|\cdot\|_{\rm {os}})$ 
which is a complete isometry, then the identity
$I:(\ell^1(3),\rm {MIN})\to (\ell^1(3),\|\cdot\|_{\rm {os}})$ 
must be also a complete isometry. 
Therefore the operator space structure $\|\cdot\|_{\texttt{os}}$ 
is different from the MIN structure. 
However, although $\|(I,U,V)\|_{\rm {MAX}}=3,$ we are unable to decide whether the operator space structure $\|\cdot\|_{\texttt{os}}$ is completely isometric to the MAX operator space structure or not. 

\textbf{Acknowledgement:} We are very grateful to G. Misra for several fruitful discussions and suggestions. We are also thankful to the referee for many valuable suggestions. 

\bibliographystyle{amsalpha}\bibliography{Embeddings_of_l1_Space}
\end{document}